\documentclass[11pt]{article}
\usepackage{layout}
\usepackage[makeroom]{cancel}
\usepackage{bbm}

\usepackage{amsmath}
\usepackage{amsthm}
\usepackage{amssymb}
\usepackage{mathtools,xparse}
\usepackage{enumerate}
\usepackage{amscd}
\usepackage{amsthm}
\usepackage{physics}
\usepackage{calligra}
\usepackage{commath}
\usepackage{scalerel}
\usepackage{color}
\usepackage{graphicx}
\usepackage[all, cmtip]{xy}
\usepackage{soul}
\usepackage{esint}
\usepackage{hyperref}
\usepackage{hypcap}
\usepackage[titletoc]{appendix}
\usepackage{verbatim}
\usepackage{amsrefs}
\usepackage{bbm}
\usepackage{mathrsfs}

\newtheorem{theorem}{Theorem}[section]

\newtheorem{assumption}[theorem]{Assumption}

\newtheorem{definition}[theorem]{Definition}
\newtheorem{proposition}[theorem]{Proposition}

\newtheorem{corollary}[theorem]{Corollary}
\newtheorem{lemma}[theorem]{Lemma}
\newtheorem{remark}[theorem]{Remark}
\newtheorem{example}[theorem]{Example}
\newtheorem{examples}[theorem]{Examples}
\newtheorem{foo}[theorem]{Remarks}

\title{A VERSION OF H\"ORMANDER'S THEOREM FOR MARKOVIAN ROUGH PATHS }

\author{Guang Yang}
\date{}

\begin{document}
	\maketitle
	
\begin{abstract}
	We consider a rough differential equation of the form \(dY_t=\sum_i V_i(Y_t)d\boldsymbol{X}^i_t+V_0(Y_t)dt \), where \(\boldsymbol{X}_t \) is a Markovian rough path. We demonstrate that if the vector fields \((V_i)_{0\leq i\leq d} \) satisfy the parabolic H\"ormander's condition, then \(Y_t\) admits a smooth density with a Gaussian type upper bound, given that the generator of \(X_t\) satisfy certain non-degenerate conditions.  The main new ingredient of this paper is the study of a non-degenerate property of the Jacobian process of \(X_t\). 
\end{abstract}
\tableofcontents

\section{Introduction}
We consider rough differential equations in $\mathbb{R}^d$ of the form
 \begin{equation}
 	\label{The general RDE}
 	dY_t=\sum_{i=1}^{d} V_i(Y_t)d\boldsymbol{X}^i_t+V_0(Y_t)dt,\; Y_0=y_0\in\mathbb{R}^d, \; t\in [0,1].
 \end{equation} 
  Over the past decade, rough differential equations driven by Gaussian processes (i.e., $\boldsymbol{X}_t$ is a Gaussian rough path) have been extensively studied. The existence and smoothness of the density of $Y_t$ is among the most important questions and have attracted a lot attentions.

  The case where $\boldsymbol{X}_t$ is given by the Stratonavich Brownian rough path is equivalent to the classical probabilistic H\"ormander's theorem studied by P. Malliavin \cite{MR536013}. In this case, it is possible to prove that $Y_t$ admits a smooth density is equivalent to the hypoellipticity of the differential operator given by
  \[L=\frac{1}{2}\sum_{i=1}^{d}V_i^2+V_0. \]
   L. H\"ormander \cite{MR222474} was the first to formulate a sufficient condition on $\{V_i \}_{0\leq i\leq d}$ to ensure the hypoellipticity of $L$, known today as the parabolic H\"ormander's condition. The parabolic H\"ormander's condition,  which we introduce now, has become a fundamental setting in many areas including probability, geometry and PDE.
 \begin{definition}
 Let $\{ V_i \}_{0\leq i\leq d}$ be a collection of smooth vector fields on $\mathbb{R}^d$. Define
 \begin{align*}
 	\mathcal{W}^0&=\big\{V_1,\cdots , V_d  \big \},\\
 	\mathcal{W}^{k+1}&=  \cup_{0\leq i\leq d}\big\{  [V,V_i]\;,\;       V\in\mathcal{W}^k \big \}\;,\; k\in \mathbb{Z}^+ .
 \end{align*}
We say $\{ V_i \}_{0\leq i\leq d}$ satisfy the parabolic H\"ormander's condition if we can find an integer \(k_0\geq 0 \) such that for every \(x\in\mathbb{R}^d \)
 \begin{equation*}
 	Span \left\{\cup_{0\leq k\leq k_0} \mathcal{W}^{k}(x) \right\} =\mathbb{R}^d.
 \end{equation*}
 \end{definition}

With the previous definition, the probabilistic H\"ormander's theorem by P. Malliavin \cite{MR536013} can be stated as follows.
\begin{theorem}
	\label{ Probabilistic Hormander's theorem}
	Let $W_t$ be a standard Brownian motion on $\mathbb{R}^d$ and consider
	\[ dY_t=\sum_{i=1}^{d} V_i(Y_t)\circ dW^i_t+V_0(Y_t)dt,\; Y_0=y_0\in\mathbb{R}^d, \; t\in [0,1]. \]
	Assume that $\{ V_i \}_{0\leq i\leq d}$ are smooth vector fields with bounded derivatives of all orders. If $\{ V_i \}_{0\leq i\leq d}$ satisfy the parabolic H\"ormander's condition, then for all $t\in (0,1]$,  $Y_t$ admits a smooth density with respect to the Lebesgue measure on $\mathbb{R}^d$. 
\end{theorem}  
  \begin{remark}
  	Theorem \ref{ Probabilistic Hormander's theorem} is still true when the Stratonovich integral is replaced by It\^o's integral in the equation. 
  \end{remark} 	
Under the same assumptions as theorem \ref{ Probabilistic Hormander's theorem} with $W_t$ replaced by a fractional Brownian motion (fBm) with Hurst parameter $H>1/2$, F. Baudoin and M. Hairer \cite{MR2322701} proved that $Y_t$ admits a smooth density. With recent developments on Gaussian rough paths, we can consider \eqref{The general RDE} with $\boldsymbol{X}_t$ given by a general non-degenerate Gaussian rough path. It was proved in \cite{MR2680405}, that the $Y_t$ admits a density if $\{ V_i \}_{0\leq i\leq d}$ satisfy the parabolic H\"ormander's condition. Smoothness of the density was proved later in \cite{MR3298472} after the tail estimate for the associated Jacobian process and a deterministic Norris's lemma were established (see \cite{MR3112937},\cite{MR3112925}).

 
 The major tool that all above-mentioned results heavily rely on is the Malliavin calculus, which is a successful application of differential measure theory to Radon Gaussian measures. Despite being sufficiently flexible to obtain related results for a number of extensions of the original problem, Malliavin calculus cannot be applied to problems without an underlying Gaussian structure. As a result, contrary to the rapid developments of its Gaussian counterpart, the study of rough differential equations driven by Markovian rough paths progresses rather slowly. In \cite{MR3814250}, I. Chevyrev and M. Ogrodnik used analysis on manifolds to prove that $Y_t$ admits a density with respect to any smooth measure with assumptions strictly stronger than parabolic H\"ormander's condition on the vector fields. The question of the smoothness of the density is still open.

 In this paper, we consider \eqref{The general RDE} with $\boldsymbol{X}_t$ given by a Markovian rough path. Our goal is to prove that, with the parabolic H\"ormander's condition assumption on the vector fields, $Y_t$ admits a smooth density with a Gaussian type upper bound. Unlike Gaussian processes, a general Markov process $X_t$ may not have a smooth density. Thus, certain regularity assumption is necessary for the coefficients of its generator. We will see very soon that this extra regularity assumption provides an underlying Gaussian structure, which, in turn, makes Malliavin calculus applicable. 
 
  Now we can introduce our basic settings. We fix two constants \(0< \lambda<\Lambda \). Let \(a(x) \) be a non-constant measurable function from \(\mathbb{R}^d \) to the space of symmetric matrices which are uniformly elliptic with respect to \(\lambda \) and bounded by \(\Lambda \), i.e.,
 \begin{equation}
 	\label{Uniform ellipticity and boundedness}
 \lambda \abs{\xi}^2 \leq \langle \xi, a(x)\xi \rangle\leq \Lambda \abs{\xi}^2
 \end{equation}
 for any \(\xi\in \mathbb{R}^d \) and almost every \(x\in\mathbb{R}^d \), where \(\abs{\xi} \) is the Euclidean norm. We use \(\Xi^{\lambda,\Lambda} \) to denote all the functions that satisfy \eqref{Uniform ellipticity and boundedness}. Define the associated differential operator
 \begin{equation}
 \label{generator}
 L=\frac{1}{2}\sum_{i,j=1}^{d} \frac{\partial}{\partial x_i}(a_{i,j}\frac{\partial}{\partial x_j} )
 \end{equation} 
 with domain \(Dom(L)=\{f\in H^2(\mathbb{R}^d) \mid Lf\in L^2 \} \). Let $\{X_t\}_{t\geq 0}$ be the Markov process generated by $L$ and define $Y_t$ as the solution to 
 \begin{equation*}
 	dY_t=\sum_{i=1}^{d} V_i(Y_t)d\boldsymbol{X}^i_t+V_0(Y_t)dt,\; Y_0=y_0\in\mathbb{R}^d, \; t\in [0,1].
 \end{equation*}
We can now state our assumptions. First two are standard.
\begin{assumption}
	\label{Assumption 1}
	The vector fields $\{V_i \}_{0\leq i \leq d}$ are smooth and bounded together with all their derivatives.
\end{assumption}
\begin{assumption}
	\label{Assumption 2}
	The vector fields $\{V_i \}_{0\leq i \leq d}$ satisfy the parabolic H\"ormander's condition.
\end{assumption}	
Our next assumption is about the regularity of $a(x)$ that we mentioned earlier.
	\begin{assumption}
		\label{Assumption 4}
		The function \(a(x) \) is smooth and bounded together with all their derivatives.
	\end{assumption}
	The smoothness assumption on $a(x)$, though seems a bit restrictive at the first look, is actually necessary. Indeed, we can let $\{V_i \}_{1\leq i\leq d}$ be the columns of the identity matrix on $\mathbb{R}^d$ and let $V_0=0$. Obviously $\{V_i\}_{0\leq i\leq d}$ satisfy \textbf{\emph{Assumptions} \ref{Assumption 1}}  and \textbf{\ref{Assumption 2}}. If we aim for a H\"ormander's type theorem we must have that $Y_t=X_t$ has a smooth density. On the other hand, we know from classical analysis that, to obtain \(k\)th-order derivative on the density of \(X_t \), one needs \(a(x) \) to be at least \((k-1) \)th-order continuously differentiable. Consequently, the smoothness of $a(x)$ is a necessary. 
	
	It is well-known that if $a(x)$ is smooth, $X_t$ is a diffusion process, whose stochastic differential equation can be written as follows:
	\begin{equation*}
		dX_t=\sum_{i=1}^{d}A_i(X_t)dW^i_t+B(X_t)dt,
	\end{equation*} 
where $A=\sqrt{a}$ and $B=\nabla \cdot a$.
 
 Central to all previous cases is the non-degeneracy of the driving signal $X_t$. It is not surprising that $A(x)$, which itself is an elliptic system, guarantees that $X_t$ is non-degenerate in certain sense. However, we will see later that, exclusive to the non-Gaussian case, the non-degeneracy of the Malliavin derivative of $X_t$ is also needed (see remark \ref{Why that assumption}). This motivates our final assumption. We want to emphasis that the next assumption is not necessary if $X_t$ generated by \eqref{generator} is a Gaussian process (e.g., when $a(x)$ is constant). Hence, it constitutes the biggest difference between non-Gaussian case and Gaussian case. 
 \begin{assumption}
	\label{Assumption 3}
	Let \(A(x)\) be the unique square root of \(a(x) \). For a fixed positive constant \(C_J\), the differential of \(A(x) \), which is a smooth map \(dA(x):\mathbb{R}^d\rightarrow \mathbb{R}^{d\times d} \), does not vanish and satisfy the following uniform non-degenerate condition
	\begin{equation*}
		\abs{v^T \cdot dA(x) }^2\geq C_J\abs{v}^2,\;\forall x,v\in \mathbb{R}^d.
	\end{equation*}
\end{assumption}


 
The main result of this paper is the following:
 \begin{theorem}
 	\label{main1}
 	Assume \( a\in\Xi^{\lambda,\Lambda} \), let \(X_t\) be the Markov process whose generator is given by \eqref{generator} with canonical rough lift \(\boldsymbol{X_t} \). Consider the rough differential equation
 	\begin{equation*}
 	Y_t=y_0+\sum_{i=1}^{d}\int_{0}^{t}V_i(Y_s)d\boldsymbol{X}^i_s+\int_{0}^{t} V_0(Y_s)ds,\;\;y_0\in\mathbb{R}^d,\;t\in[0,1].
 	\end{equation*}
 	Suppose that $a(x)$ is \textbf{not constant} and \textbf{Assumptions \ref{Assumption 1}, \ref{Assumption 2}, \ref{Assumption 4}}  and \textbf{\ref{Assumption 3}} are satisfied. Then for any \(t\in(0,1] \) , \(Y_t\) has a smooth density \(p_{Y_t}(y) \) with respect to the Lebesgue measure on \(\mathbb{R}^d \). Moreover, \(p_{Y_t}(y) \) has the following Gaussian type upper bound,
 	\begin{equation*}
 	p_{Y_t}(y)\leq C_1(t)\exp(-\frac{C_2 (y-y_0)^2}{t}).
 	\end{equation*}
 \end{theorem}
	The fundamental argument for our proof is again the classical Malliavin calculus adapted to the rough paths theory, which involves the study of the Malliavin derivative of $Y_t$ and a small ball estimate for $X_t$. The main difficulty in the our case is the extra integral structure in the Malliavin derivative of $Y_t$, which never appeared in any previous cases as explained in remark \ref{Difference}. We will tackle this by developing a non-degenerate property of the Jacobian process of $X_t$.
	
	There is another important point that we would like to emphasis. We shall see later that the majority of our proofs can be completely carried out in the language of classical stochastic calculus. This is due to the fact that $X_t$ and its Malliavin derivative are both diffusion and rough integrals against diffusion coincide with the corresponding Stratonovich integrals. We will frequently take advantage of this fact in our proofs. From this perspective, our results can be viewed as an application of stochastic calculus to a rough path problem. However, we will continue to present our results in the rough paths setting. The reason is twofold. On one hand, formula \eqref{Malliavin derivative}, which is crucial for future studies in more general settings, can only be generalized as rough integrals. On the other hand, results like proposition \ref{Norris} are, in fact, deterministic and point-wise, which are stronger than their stochastic counterparts. We prefer to keep them in this stronger form.  


\section{Preliminary material}

\subsection{Rough paths}
For $\alpha\in (\frac{1}{3},\frac{1}{2}]$, we define the space of $\alpha$-H\"older rough path on $\mathbb{R}^d$, in symbols $\mathscr{C}^\alpha([0,1], \mathbb{R}^d)$, as those pairs $(X, \mathbb{X})=: \boldsymbol{X}\in C([0,1], \mathbb{R}^d\oplus (\mathbb{R}^d)^{\otimes 2})$ such that
\begin{equation*}
	\norm{X}_{\alpha}=\sup_{s\neq t\in[0,1]}\frac{\abs{X_{s,t}}}{\abs{t-s}^\alpha}<+\infty,\; \; \norm{\mathbb{X}}_{2\alpha}=\sup_{s \neq t\in [0,1]}\frac{\abs{\mathbb{X}_{s,t}}}{\abs{t-s}^{2\alpha}}<+\infty,
\end{equation*}
where $X_{s,t}=X_t-X_s$ and similarly for $\mathbb{X}_{s,t}$. Moreover, for $0\leq s\leq u \leq t \leq 1$, 
\begin{equation*}
	\mathbb{X}_{s,t}=\mathbb{X}_{s,u}+\mathbb{X}_{u,t}+X_{s,u}\otimes X_{u,t}.
\end{equation*}
We equip $\mathscr{C}^\alpha([0,1], \mathbb{R}^d)$ with the rough metric
\begin{equation*}
	\rho_{\alpha}(\boldsymbol{X}, \boldsymbol{Y})=\sup_{s\neq t\in[0,1]}\frac{\abs{X_{s,t}-Y_{s,t}}}{\abs{t-s}^\alpha}+\sup_{s \neq t\in [0,1]}\frac{\abs{\mathbb{X}_{s,t}-\mathbb{Y}_{s,t}}}{\abs{t-s}^{2\alpha}},
\end{equation*}
and define $\rho_{\alpha}(\boldsymbol{X}):=\rho_{\alpha}(\boldsymbol{X},0 )$. Note that if $X\in BV([0,1];\mathbb{R}^d)$, $\mathbb{X}$ can be canonically defined as
\begin{equation*}
	\mathbb{X}_{s,t}=\int_{s}^{t}X_{s,r}dX_r,
\end{equation*}
where the integral is understood as Riemann–Stieltjes integral and $(X,\mathbb{X}) \in\mathscr{C}^\alpha([0,1], \mathbb{R}^d)$ for any $\alpha\in (0,1]$.  A rough path $\boldsymbol{X}\in \mathscr{C}^\alpha([0,1], \mathbb{R}^d)$ is said to be a geometric $\alpha$-H\"older rough path if we can find a sequence $\{X^k\}_{k\geq 1}\in BV([0,1];\mathbb{R}^d)$ such that
\begin{equation}
	\label{Geometric rough path}
	\lim_{k\rightarrow \infty}\rho_\alpha(\boldsymbol{X},\boldsymbol{X}^k)\rightarrow 0.
\end{equation}

Another important notion is the so called controlled rough path. Let $\alpha\in (\frac{1}{3}, \frac{1}{2}]$ and $X_t\in C^\alpha([0,1], \mathbb{R}^d)$ be an $\alpha$-H\"older continuous path. For a Banach space $W$, we say $Y_t\in  C^\alpha([0,1], W)$, is controlled by $X_t$ if we can find $Y'_t\in C^\alpha([0,1], L(\mathbb{R}^{d}, W)) $ such that the remainder term $R^Y$ given by
\begin{equation*}
	Y_{s,t}=Y'_sX_{s,t}+R^Y_{s,t},
\end{equation*}
satisfies $\norm{R^Y}_{2\alpha}<+\infty$, where $L(\mathbb{R}^{d}, W)$ is the space of all linear operators from $\mathbb{R}^{d}$ to $W$. We denote the space of all $W$-valued controlled paths $\mathscr{D}^{2\alpha}_X(W)$, and endow it with the semi-norm
\[\norm{Y ,Y'}_{X, 2\alpha}=\norm{Y'}_{\alpha}+\norm{R^Y}_{2\alpha}.   \]
$\mathscr{D}^{2\alpha}_X$ becomes a Banach space with the norm \((Y,Y')\mapsto \abs{Y_0}+\abs{Y'_0}+\norm{Y ,Y'}_{X, 2\alpha}\). Controlled paths are stable under composition with with regular functions. In fact, it is easy to check that if $f\in C^2_b(W)$ , then $f(Y_t)\in \mathscr{D}^{2\alpha}_X$ with \[f(Y_t)'=f'(Y_t)Y_t'. \]

If $Y_t$ is controlled by $X_t$ with $W=L(\mathbb{R}^d, V)$ for a Banach space $V$. Then the rough integral of $Y_t$ against $X_t$ is defined as
\begin{equation*}
	\int_{0}^{t}Y_sd\boldsymbol{X}_s:=\lim_{{\Delta_n}\rightarrow 0}\sum_{i}Y_{t_{n_i}}\cdot X_{t_{n_i},t_{n_{i+1}}}+Y'_{t_{n_i}}\cdot\mathbb{X}_{t_{n_i},t_{n_{i+1}}},
\end{equation*}
where $\{t_{n_i} \}_{i\geq 1}$ is a sequence of partitions of $[0,t]$ with mesh $\Delta_n$. Note that we use the canonical injection $L(\mathbb{R}^{d}, L(\mathbb{R}^{d}, V))\xhookrightarrow {} L(\mathbb{R}^{d}\otimes\mathbb{R}^d, V) $ in writing $Y'_{t_{n_i}}\cdot\mathbb{X}_{t_{n_i},t_{n_{i+1}}}$. The rough integral can be seen as a map from $\mathscr{D}^{2\alpha}_X(W)$ to $\mathscr{D}^{2\alpha}_X(V)$. Since
\[\left(\int_{0}^{t}Y_sd\boldsymbol{X}_s, Y_t\right)\in \mathscr{D}^{2\alpha}_X(V).  \]
This map is in fact continuous and we have the estimate
\begin{equation*}
	\norm{\left(\int_{0}^{t}Y_sd\boldsymbol{X}_s, Y_t\right)}_{X,2\alpha}\leq \norm{Y}_\alpha+\norm{Y'_t}_{L^\infty}\norm{\mathbb{X}}_{2\alpha}+C\left(\norm{X}_{\alpha}\norm{R^Y}_{2\alpha}+\norm{Y'_t}_{\alpha}\norm{\mathbb{X}}_{2\alpha}\right),
\end{equation*}
where $C$ is a positive constant depends on $\alpha$.

Finally, let $f\in C^2_b(V,L(\mathbb{R}^{d}, V))$. We may consider the rough differential equation driven by $\boldsymbol{X}$ given by
\begin{equation}
	\label{Equation for Z}
	Z_t=z_0+\int_{0}^{t}f(Z_s)d\boldsymbol{X}_s.
\end{equation} 
A process $Z_t\in V$ is said to be a solution \eqref{Equation for Z} if $Z_t\in\mathscr{D}^{2\alpha}_X(V)$ and the integral in \eqref{Equation for Z} holds as a rough integral. We may also consider more general equations like
\begin{equation*}
	dY_t=\sum_{i=1}^{d} V_i(Y_t)d\boldsymbol{X}^i_t+V_0(Y_t)dt,\; Y_0=y_0\in\mathbb{R}^d, \; t\in [0,1].
\end{equation*}
For simplicity, We will follow the usual convention and write
\[Y_t=\pi_V(0,y_0; \boldsymbol{X})(t), \ t\in[0,1]. \]
The map $\pi$ is called the It\^o-Lyons map. Moreover, if $\boldsymbol{X}$ is a geometric $\alpha$-H\"older rough path, then we have
\begin{equation*}
	\lim_{k\rightarrow \infty}\norm{\pi_V(0,y_0; \boldsymbol{X})-\pi_V(0,y_0; \boldsymbol{X}^k)}_\alpha=0,
\end{equation*} 
where $\{\boldsymbol{X}^k\}_{k\geq 1}$ is any sequence such that \eqref{Geometric rough path} is satisfied.
\subsection{Markovian rough paths}

Central to our purpose is the Markovian rough paths. Let $X_t$ be the Markov process generated by \eqref{generator}. For Let $t\in(0,1]$ and \(\{D_n \}_{n\geq 1} \) be a sequence of increasing partitions of interval \([0,t]\) with $\Delta_n\rightarrow 0$. Define
\begin{equation*}
	K^n_{i,j}(X)_t=\sum_{t_k\in D_n}\frac{X^i_{t_{k+1}}+X^i_{t_k}}{2}(X^j_{t_{k+1}}-X^j_{t_k}),
\end{equation*} 
then \(K_{i,j}(X)_t:=\lim_{n\rightarrow \infty} K^n_{i,j}(X)_t \) exists in probability, and the couple \(\boldsymbol{X}_t=(X_t, K_t ) \) is a geometric rough path in \(\mathscr{C}^\alpha([0,1], \mathbb{R}^d) \) for any \(\alpha\in(0,1/2) \) (see \cite{MR2247926}). \(\boldsymbol{X}_t\) is called the canonical rough lift of \(X_t\).

\begin{remark}
	It is worth mentioning that there are other equivalent ways to construct the rough lift of $X_t$, see for example \cite{MR2438699}. 
\end{remark}

\subsection{Malliavin calculus}

We collect some basic materials in Malliavin calculus and refer to \cite{MR2200233} for a complete exploration. 

Let $\mathcal{F}_t$ be the filtration generated by a Brownian motion $W_t$ and $\mathcal{H}=W^{1,2}_0([0,1])$ be the Cameron-Martin space of $W_t$. A $\mathcal{F}_1$-measurable random variable $F$ is said to be cylindrical if it has the form
\begin{equation*}
	F=f(W_{t_1}, W_{t_2}, \cdots ,W_{t_n}),
\end{equation*}
where $f\in C^\infty_b(\mathbb{R}^n, \mathbb{R})$ and $\{t_i\}_{1\leq i\leq n}\in [0,1]$. We denote the collection of all cylindrical random variables $\mathcal{S}$. 

The Malliavin derivative of $F$ is defined as
\begin{equation*}
	DF=\sum_{i=1}^{n}\partial_if(W_{t_1}, W_{t_2}, \cdots ,W_{t_n})1_{[0,t_i]},
\end{equation*}
where $1_{[0,t_i]}$ is the indicator function of interval $[0,t_i]$. For $h\in\mathcal{H}$, we have the following relation between directional derivatives of $F$ and $DF$:
\begin{equation*}
D_hF:=\langle DF, \dot{h} \rangle_{L^2([0,1];\mathbb{R}^d)}=\lim_{\epsilon\rightarrow 0}\frac{f(W_{t_1}+\epsilon h_{t_1}, \cdots,W_{t_n}+\epsilon h_{t_n})-f(W_{t_1}, \cdots ,W_{t_n})}{\epsilon},
\end{equation*}
where $\dot{h}$ is the derivative of $h$. It is often useful to use the canonical isometry between $\mathcal{H}$ and $L^2([0,1])$ and take $DF$ as an element in $\mathcal{H}$. (In fact, on abstract Wiener spaces,  Malliavin derivatives are defined to be random variables in the corresponding Cameron-Martin space; the usual setting for Brownian motion is a very special case.) One could then iterate the previous definition the define the $n$-th Malliavin derivative of $D^nF$ which takes value in $\mathcal{H}^{\otimes n}$. 

For any $p\geq 1$, it is possible to prove that $D^n$ is closable from $\mathcal{S}$ to $L^p(\Omega; \mathcal{H}^{\otimes n})$. We denote by $\mathbb{D}^{n,p}$ the closure of $\mathcal{S}$ with respect to the norm
\begin{equation*}
	\norm{F}_{n,p}=\left( \mathbb{E}\abs{F}^p+\sum_{i=1}^{n}\mathbb{E}\norm{D^iF}_{\mathcal{H}^{\otimes i}}^p        \right)^{\frac{1}{p}},  
\end{equation*}
and $\mathbb{D}^{\infty}=\cap_{n\geq 1}\cap_{p\geq 1}\mathbb{D}^{n,p}$. 

If the vector fields $\{A_i\}_{1\leq i\leq d}, B\in C^\infty_b$, then for any $t\in [0,1]$, $X_t$ defined by
\begin{equation*}
	dX_t=\sum_{i=1}^{d}A_i(X_t)dW^i_t+B(X_t)dt, X_0=x_0\in\mathbb{R}^d
\end{equation*}
belongs to $\mathbb{D}^{\infty}$. Moreover, let $A=(A_1,A_2,\cdots,A_d)$ then $DX_t$ satisfies
\begin{equation*}
	D_rX_t=A(X_r)+\int_{0}^{t}DA_i(X_s)\cdot D_rX_sdW^i_s+\int_{0}^{t}DB(X_s)\cdot D_rX_sds,
\end{equation*}
for $r\leq t$ and $D_rX_t=0$ for $s>t$ a.s. 

Recall the Jacobian process associated with $X_t$ is given by
\begin{equation*}
	J^X_{t\leftarrow 0}=I_{d\times d}+\sum_{i=1}^{d}\int_{0}^{t} DA_i(X_s)\cdot J^X_{s\leftarrow 0}dW^i_s+\int_{0}^{t}DB(X_s)\cdot J^X_{s\leftarrow 0}ds,\ t\in [0,1].
\end{equation*}
The inverse of $J^X_{t\leftarrow 0}$, written as $J^X_{0\leftarrow t}$, is well defined, and we have the following composition property:
\[J^X_{t\leftarrow u}\cdot J^X_{u \leftarrow s}=J^X_{t\leftarrow s}, \]
for $s,u,t\in [0,1]$. The following integrability result is well-known and we refer to theorem 7.2 and remark 7.3 of \cite{MR3298472} for a proof. 
\begin{proposition}
	\label{Integrability}
	Suppose that $\{A_i\}_{1\leq i\leq d}, B\in C^\infty_b(\mathbb{R}^d)$. Define $\Phi_t=(X_t, J^X_{t\leftarrow 0}, J^X_{0\leftarrow t})$. Then for any $\gamma\in(0,\frac{1}{2})$, we have $\norm{\Phi}_\gamma\in L^p(\Omega)$ for any $p\geq 1$. In particular, $\norm{\Phi}_{\infty}=\sup_{t\in[0,1]}\abs{\Phi(t)}\in L^p(\Omega)$ for any $p\geq 1$.
\end{proposition}

By uniqueness, it is straightforward to see that, for any $r\leq t$ we have
\begin{equation}
	\label{Malliavin derivative for X_t}
	D_rX_t=J^X_{t \leftarrow r}A(X_r).
\end{equation}

Another fundamental object in Malliavin calculus is given by the next
\begin{definition}
	Let $F=(F^1, F^2, \cdots, F^d)$ be a random vector with whose components belong to $\mathbb{D}^{1,1}$. The Malliavin matrix $\Gamma(F)$ is defined as
	\begin{equation*}
		\Gamma(F)=(\langle DF^i,DF^j \rangle_{\mathcal{H}})_{1\leq i,j\leq d}.
	\end{equation*}
\end{definition}
An estimate on the Malliavin matrix allows one to prove the existence and smoothness of density.
\begin{proposition}
	\label{Smoothness}
	Let $F=(F^1, F^2, \cdots, F^d)$ be a random vector whose components belong to $\mathbb{D}^{\infty}$. If $(\det \Gamma(F) )^{-1}\in L^p(\Omega)$ for all $p\geq 1$, then
	$F$ has a smooth density with respect to the Lebesgue measure on $\mathbb{R}^d$.
\end{proposition}

\section{Malliavin derivative of $Y_t$}

This section is devoted to the computation of the Malliavin derivative of \(Y_t\) defined in theorem \ref{main1}. Let \(\{D_n \}_{n\geq 1} \) be a sequence of increasing partitions of \([0,t] \), with the mesh $\Delta_n\rightarrow 0$. For any process \(Z \) defined on \([0,t] \), we use \(Z^n \) to represent the piece-wise linear approximation of \(Z\) along \(D_n \).  

Similar to $X_t$, the associated Jacobian process $J^Y_{t\leftarrow 0}$ is a $d\times d$ matrix-valued process given by
\begin{equation*}
dJ^Y_{t\leftarrow 0}=\sum_{i=1}^{d} DV_i(Y_t)J^Y_{t\leftarrow 0}d\boldsymbol{X}^i_t+DV_0(Y_t)J^Y_{t\leftarrow 0}dt,\; J^Y_{0\leftarrow 0}=I_{d\times d}, \; t\in [0,1].
\end{equation*}

Now let us write
\begin{equation*}
Y_t(\omega)=y_0+\sum_{i=1}^{d}\int_{0}^{t}V_i(Y_s)d \boldsymbol{X}_s^i(\omega)+\int_{0}^{t}V_0(Y_s)ds,\; y_0\in\mathbb{R}^d,  \; t\in[0,1],
\end{equation*}
where \(\omega \) represents the underlying Brownian paths. For every \( h \in\mathcal{H} \), we have by definition 
\begin{equation}
\label{definition of DY}
D_hY^j_t=\langle DY^j_t,\dot{h} \rangle_{L^2([0,1]) }=\lim_{\epsilon\rightarrow 0} \frac{Y^j_t(\omega+\epsilon h)-Y^j_t(\omega)}{\epsilon} ,\;\; 1\leq j\leq d.
\end{equation}
Here 
\begin{equation*}
	Y_t(\omega+\epsilon h)=y_0+\int_{0}^{t}V_0(Y_s(\omega+\epsilon h))ds+\sum_{i=1}^{d}\int_{0}^{t}V_i(Y_s(\omega+\epsilon h)) d\boldsymbol{X}(\omega+\epsilon h)^i_s.
\end{equation*}
Note that \(X(\omega+\epsilon h) \) is the diffusion given by
\begin{equation*}
X_t(\omega+\epsilon h)=x_0+\sum_{i=1}^{d} \int_{0}^{t}A_i(X_s(\omega+\epsilon h))d(W+\epsilon h)^i_s+\int_{0}^{t}B(X_s(\omega+\epsilon h))ds.
\end{equation*}
Hence, the rough lift $\boldsymbol{X}(\omega+\epsilon h)_t$ is well defined.
	\begin{proposition}
	Let \(J^Y_{t\leftarrow0 }\) be the Jacobian process of \(Y_t\). Under assumptions of theorem \ref{main1} and for all \( t\in [0,1] \), \(h\in \mathcal{H}\), we have 
	\begin{equation*}
		D_hY_t=\sum_{i=1}^{d}J^Y_{t\leftarrow 0 }  \int_{0}^{t}J^Y_{0\leftarrow s } V_i(Y_s )d \boldsymbol{D_hX}^i_s.
	\end{equation*}
	almost surely.
\end{proposition}
\begin{remark}
	Although this result is straightforward, to the author's best knowledge it is not written in any standard reference. We include a proof for completeness. 
\end{remark}

\begin{proof}
	We first note that, $X^n$ has bounded variation almost surely. Thus, $\pi_V(0,y_0;\boldsymbol{X^n})$ is just the classical It\^o map and we will simply write it as $\pi_V(0,y_0;X^n)$.	We know that $\pi_V(0,y_0;\boldsymbol{X})=\lim_{n\rightarrow \infty}\pi_V(0,y_0;X^n)$ uniformly. As a result
	\begin{equation*}
		\lim_{\epsilon\rightarrow0}\frac{Y_t(\omega+\epsilon h)-Y_t(\omega)}{\epsilon}=\lim_{\epsilon\rightarrow0}\lim_{n\rightarrow \infty}\frac{\pi_V(0,y_0;(X+\epsilon h)^n)(t)-\pi_V(0,y_0;X^n)(t)}{\epsilon}.
	\end{equation*}
	On the other hand, we have
\[  \lim_{\epsilon\rightarrow 0}\frac{(X+\epsilon h)^n-X^n}{\epsilon}=D_hX^n.\]
By Duhamel's principle of ODE (see for example \cite{friz2010multidimensional} section 4.1 and 4.2), we have
\begin{equation*}
	\lim_{\epsilon\rightarrow 0}\frac{\pi_V(0,y_0;(X+\epsilon h)^n)(t)-\pi_V(0,y_0;X^n)(t)}{\epsilon}=\sum_{i=1}^{d}\int_{0}^{t}J^{Y_n}_{t\leftarrow s} V_i(Y_n(s))d (D_hX^{n})^i_s,
\end{equation*}
where $Y_n(t)=\pi_V(0,y_0;X^n)(t)$. Finally, sending $n$ to infinity gives
\begin{equation*}
	\lim_{\epsilon\rightarrow0}\frac{Y_t(\omega+\epsilon h)-Y_t(\omega)}{\epsilon}=\sum_{i=1}^{d}\int_{0}^{t}J^{Y}_{t\leftarrow s} V_i(Y_s)d \boldsymbol{D_hX}^i_s.
\end{equation*}
\end{proof}

\begin{corollary}
	Let \(J^Y_{t\leftarrow0 }\) be the Jacobian process of \(Y_t\). Under assumptions of theorem \ref{main1} and for all \( t\in [0,1] \) and $1\leq j\leq d$, we have 
	\begin{equation}
		\label{Malliavin derivative}
D^j_rY_t=\sum_{i=1}^{d}J^Y_{t\leftarrow 0 } \int_{0}^{t}J^Y_{0\leftarrow s } V_i(Y_s )d\boldsymbol{D}^j_r\boldsymbol{X}^i_s.
	\end{equation}
\end{corollary}
\begin{proof}
	By definition we have 
	\begin{equation*}
		D_hX^i_t=\sum_{j=1}^{d}\langle D^j_rX^i_t, \dot{h}^j_r \rangle_{L^2([0,1])}.
	\end{equation*}
		With same notations as previous proposition, we have
		\begin{equation*}
			\sum_{i=1}^{d}\int_{0}^{t}J^{Y_n}_{t\leftarrow s} V_i(Y_n(s))d (D_hX^{n})^i_s=\sum_{j=1}^{d}\langle \sum_{i=1}^{d}\int_{0}^{t}J^{Y_n}_{t\leftarrow s} V_i(Y_n(s))d (D^j_rX^{n})^i_s, \dot{h}^j_r \rangle_{L^2([0,1])}.
		\end{equation*}
	Taking the limit and we immediately have
	\begin{equation*}
		D^j_rY_t=\sum_{i=1}^{d}J^Y_{t\leftarrow 0 } \int_{0}^{t}J^Y_{0\leftarrow s } V_i(Y_s )d\boldsymbol{D}^j_r\boldsymbol{X}^i_s.
	\end{equation*}
\end{proof}

\begin{remark}
	\label{Difference}
	We can make a quick comparison between \eqref{Malliavin derivative} and \eqref{Malliavin derivative for X_t}. Note that if \(a(x)\equiv I_{d\times d} \), then \(X_t\) is a standard Brownian motion and \(D_rX_t=I_{d\times d} \cdot\boldsymbol{1}_{[0,t]}(r) \). If we fix $r$, $D_rX_t$ is diagonal and a pure jump process with a jump at $r$. By \eqref{Malliavin derivative} we have
	\begin{equation}
		\label{Gaussian special}
		D^j_rY_t=\sum_{i=1}^{d}J^Y_{t\leftarrow 0 } \int_{0}^{t}J^Y_{0\leftarrow s } V_i(Y_s )d\boldsymbol{D}^j_r\boldsymbol{X}^i_s=J^Y_{t\leftarrow r } V_j(Y_r ),
	\end{equation}
which recovers \eqref{Malliavin derivative for X_t}. We see that in \eqref{Gaussian special}, the rough integral degenerates into evaluation at a single point and this happens when \(X_t\) is Gaussian; because Malliavin derivatives of Gaussian processes are always pure jump processes when $r$ is fixed. Our goal is to study the more general rough integral formula \eqref{Malliavin derivative}; that is why we assume that $a(x)$ is not constant. 
\end{remark}
\begin{remark}
	\label{Why that assumption}
	Now is a good time tp further discuss the motivation of \textbf{\emph{Assumptions} \ref{Assumption 3}}. One crucial step in the proof of a H\"ormander's type theorem is the implication that 
	\begin{equation}
		\label{Important step}
		\left\{ \inf_{\norm{v}= 1}v^T\cdot \langle DY_t, DY_t \rangle_{\mathcal{H}}\cdot v\leq \epsilon  \right\}\Rightarrow \left\{ \norm{J^Y_{0\leftarrow s }V_i(Y_s )}_\infty\leq \epsilon^\alpha \right\},
	\end{equation}
	for some $\alpha>0$ and $1\leq i \leq d$. When $X_t$ is Brownian motion, this step was done using a non-property of the $L^2$ norm (see lemma A.3 of \cite{MR2814425}). For the case where $X_t$ is a non-degenerate Gaussian process, this is done by an interpolation inequality (see theorem 6.9 of \cite{MR3298472}). Since we have $D_rY_t=J^Y_{t\leftarrow r }V(Y_r )$ when $X_t$ is Gaussian, in all previous cases \eqref{Important step} can be roughly understood as
	\begin{equation*}
		\left\{ \inf_{\norm{v}= 1}v^T\cdot \langle DY_t, DY_t \rangle_{\mathcal{H}}\cdot v\leq \epsilon  \right\}\Rightarrow \left\{\norm{DY_t}_\infty\leq \epsilon^\alpha \right\}.
	\end{equation*} 
	
	However, in our case, due to the integral representation \eqref{Malliavin derivative}, same type of argument only gives
	\begin{equation*}
		\left\{ \inf_{\norm{v}= 1}v^T\cdot \langle DY_t, DY_t \rangle_{\mathcal{H}}\cdot v\leq \epsilon  \right\}\Rightarrow \left\{\norm{\int_{0}^{t}J^Y_{0\leftarrow s }V_i(Y_s )d\boldsymbol{D}_r\boldsymbol{X}^i_s}_\infty \leq \epsilon^\beta  \right\},
	\end{equation*} 
	for some $\beta>0$, $1\leq i\leq d$. We still need to further justify that 
	\begin{equation}
		\label{General case}
		 \left\{ \norm{\int_{0}^{t}J^Y_{0\leftarrow s }V_i(Y_s )d\boldsymbol{D}_r\boldsymbol{X}^i_s}_\infty \leq \epsilon^\beta \right\} \Rightarrow \left\{ \norm{J^Y_{0\leftarrow s }V_i(Y_s )}_\infty\leq \epsilon^\alpha \right\}.
	\end{equation}
	To prove \eqref{General case}, it requires \(D_rX_s=J^X_{s\leftarrow r } A(X_r ) \) to be non-degenerate. Since $A$ is an elliptic system, the behavior of $D_rX_s$ essentially depends on the Jacobian process of \(X_t\). Recall 
	\begin{equation*}
		dJ^X_{t\leftarrow 0}=\sum_{i=1}^{d}DA_i(X_t)J^X_{t\leftarrow 0}dW^i_t+DB(X_t)J^X_{t\leftarrow 0}dt.
	\end{equation*}
	Thus, a control on $(DA_1,DA_2, \cdots ,DA_d)=dA$ will be sufficient to give the desired non-degeneracy. From this perspective, \textbf{\emph{Assumptions} \ref{Assumption 3}} is very natural.  
\end{remark}


\section{Small ball estimate}
This section is devoted to developing technical tools necessary for the proof of our main result.

  \subsection{H\"older roughness of diffusion}
  
  Norris' type lemmas are crucial in the proofs of H\"ormander's type theorems. In classical stochastic analysis theory, it is nothing but a quantitative version of the Doob-Meyer decomposition of semi-martingales. In the context of rough differential equations, we are going to use a deterministic version of Norris' lemma, which first appeared in \cite{MR3112925} and was improved in \cite{MR3298472}. We start with a definition.
  \begin{definition}
  	Let \(\theta\in(0,1) \). A path \(X:[0,1]\rightarrow \mathbb{R}^d \) is called \(\theta \)-H\"older rough if there exists a constant \(c > 0\) such that for every \(s \) in \([0,1] \), every \(\epsilon \) in \((0,\frac{1}{2}] \), and every \(\phi\in \mathbb{R}^d \) with \(\abs{\phi}=1 \), there exists \(t\) in \([0,1] \) such that \(\epsilon/2 < \abs{t-s}<\epsilon  \) and
  	\begin{equation*}
  		\abs{\langle \phi, X_{s,t} \rangle }>c\epsilon^\theta.
  	\end{equation*}
  	The largest such constant is called the modulus of \(\theta \)-roughness, and is denoted by \(L_\theta(X)\).
  \end{definition}
  Now we can state the Norris type result that we are going to use and refer to \cite{MR3298472} for a proof. 
  
  \begin{proposition}
  	\label{Norris}
  	Let \(\boldsymbol{X} \) be a geometric $\alpha$-H\"older rough path and assume that \(X \) is a \(\theta \)-H\"{o}lder rough with \(2\alpha>\theta \). Let \(Y\) be a  \(\mathbb{R}^d \)-valued path controlled by $X$, and set
  	\begin{equation*}
  		Z_t=\sum_{i=1}^{d} \int_{0}^{t}Y^i_sd\boldsymbol{X}^i_s+\int_{0}^{t}b_sds,
  	\end{equation*}
  	where $b$ is an $\alpha$-H\"older continuous function. Then there exists constants \(l>0\) and \(q>0\) such that, setting
  	\begin{equation*}
  		\mathcal{A}:=1+L_{\theta}(X)^{-1}+\rho_{\alpha}(\boldsymbol{X})+\norm{Y }_{\alpha }+\norm{Y'}_{\alpha}+\norm{b}_{\alpha},
  	\end{equation*}
  	one has the bound
  	\begin{equation*}
  		\norm{Y}_{\infty}+\norm{b}_{\infty}\leq M\mathcal{A}^q\norm{Z}^l_{\infty}
  	\end{equation*}
  	where \(M\) depends on the dimension \(d \) and \(Y\).
  \end{proposition}
   We aim to prove H\"older roughness for \(X_t\) in this subsection. Let \(v\in\mathbb{R}^d\) with \(\abs{v}=1 \), then \( v^TX_{s,t}  \) is given by
\begin{equation*}
v^TX_{s,t}=\int_{s}^{t} v^T A_i(X_l) dW^i_l+\int_{s}^{t} v^T B(X_l)  dl ,\  \forall t \in [s,1].
\end{equation*}	
We are able to prove the following small ball estimate.
\begin{lemma}
	\label{smallball}
	If \(\{A_i \}_{1\leq i \leq d} \) form an elliptic system, then for any $s\in [0,1)$ and \( k\in(0,1) \), we can find \(\epsilon_0>0 \) and constants $C, C'>0$ such that for all \(0<\epsilon<\epsilon_0 \) and $\delta\in (0,1-s)$ 
	\begin{equation*}
	\mathbb{P}\left(\inf_{\norm{v}= 1} \sup_{t\in[s,s+\delta]}  \abs{v^T X_{s,t}} \leq \epsilon \right)\leq  C\exp\left\{-C'\frac{ \delta}{\epsilon^{2-2k} }\right\}.
	\end{equation*}
\end{lemma}	
\begin{proof}
	For a fixed \(v\in\mathbb{R}^d \) and \( s\in[0,1) \) , define \(M_{s,r}(v)= v^TX_{s,r}  \) for $s\leq r$. Then
	\begin{equation*}
	\sup_{t\in[s,s+\delta]}  \abs{v^T X_{s,t}}=\sup_{r\in [s,s+\delta] } \abs{M_{s,r}}.
	\end{equation*}	
	Let $M^m_{s,r}$ be the local martingale part of $M_{s,r}$. By uniform ellipticity, we can find a constant \(C_1>0\) such that
	\begin{equation}
		\label{Ellipticity}
	\langle M^m_{s,r} , M^m_{s,r} \rangle_{s,s+\delta}\geq\delta \inf_{r\in[s,s+\delta]} \sum_{i=1}^{d} \abs{v^T A_i(X_r)} ^2 \geq  \delta C_1, 
	\end{equation}
	where the bracket means the quadratic variation. We can therefore deduce that 
	\begin{equation*}
		\mathbb{P}\left(\sup_{r\in [s,s+\delta] } \abs{M_{s,r}}\leq \epsilon \right)=\mathbb{P}\left(\sup_{r\in [s,s+\delta] } \abs{M_{s,r}}\leq \epsilon,\ \langle M^m_{s,r} , M^m_{s,r} \rangle_{s,s+\delta}\geq  C_1 \delta   \right).
	\end{equation*}
	By It\^o's formula, we have
	\begin{equation}
		\label{Ito}
		M^2_{s,r}=2\sum_{i=1}^{d}\int_{s}^{r}M_{s,l} \cdot v^T A_i(X_l)  dW^i_l+2\int_{s}^{r} M_{s,l}\cdot v^T B(X_l)  dl+\sum_{i=1}^{d} \int_{s}^{r} \abs{v^T A_i(X_l)}^2dl.
	\end{equation}
	Since $\{A_i\}_{1\leq i\leq d}, B\in C^\infty_b$, we can find $C_B>0$ such that
	\begin{equation}
		\label{First implication}
	\left\{\sup_{r\in [s,s+\delta] } \abs{M_{s,r}}\leq \epsilon \right\} \Rightarrow \left\{\sup_{r\in [s,s+\delta] } \abs{\int_{s}^{r} M_{s,l}\cdot v^T B(X_l) dl}  \leq \delta C_B \epsilon  \right\}\ \& \  \left\{ \int_{s}^{s+\delta}M_{s,l}^2 \abs{v^T A_i(X_l)}^2 dl\leq \delta C^2_B \epsilon^2   \right\}.
	\end{equation}	
	By the second term of the last implication we have the following decomposition for $k\in (0,1)$ and $1\leq i\leq d$
	\begin{align}
	\mathbb{P}\left( \sup_{r\in [s,s+\delta] } \abs{M_{s,r}}^2\leq\epsilon^2 \right)=\mathbb{P}\left( \sup_{r\in [s,s+\delta] } \abs{M_{s,r}}^2\leq\epsilon^2, \int_{s}^{s+\delta}M_{s,l}^2 \abs{v^T A_i(X_l)}^2 dl\leq \delta C^2_B \epsilon^2      \right)\nonumber \\
	=\mathbb{P} \left(\sup_{r\in [s,s+\delta] } \abs{M_{s,r}}^2\leq\epsilon^2, \int_{s}^{s+\delta}M_{s,l}^2 \abs{v^T A_i(X_l)}^2 dl\leq \delta C^2_B \epsilon^2 , \sup_{r\in [s,s+\delta] } \abs{\int_{s}^{r}M_{s,l}\cdot  v^T A_i(X_l) dW^i_l} > \delta \epsilon^k   \right)\label{First} \\
	+\mathbb{P}\left( \sup_{r\in [s,s+\delta] } \abs{M_{s,r}}^2\leq\epsilon^2, \int_{s}^{s+\delta}M_{s,l}^2 \abs{v^T A_i(X_l)}^2 dl\leq \delta C^2_B \epsilon^2 , \sup_{r\in [s,s+\delta] } \abs{\int_{s}^{r}M_{s,l}\cdot  v^T A_i(X_l) dW^i_l} \leq\delta \epsilon^k  \right). \label{Second} 
	\end{align}
	For \eqref{First}, we have by the exponential inequality for martingales (see, \cite{revuz2013continuous}, p. 153 ) that
	\begin{align*}
	\mathbb{P}& \left(\sup_{r\in [s,s+\delta] } \abs{M_{s,r}}^2\leq\epsilon^2, \int_{s}^{s+\delta}M_{s,l}^2 \abs{v^T A_i(X_l)}^2 dl\leq \delta C^2_B \epsilon^2 , \sup_{r\in [s,s+\delta] } \abs{\int_{s}^{r}M_{s,l}\cdot  v^T A_i(X_l) dW^i_l} > \delta \epsilon^k   \right) \\
	&\leq \mathbb{P}\left(  \int_{s}^{s+\delta}M_{s,l}^2 \abs{v^T A_i(X_l)}^2 dl\leq \delta C^2_B \epsilon^2 , \sup_{r\in [s,s+\delta] } \abs{\int_{s}^{r}M_{s,l}\cdot  v^T A_i(X_l) dW^i_l} > \delta \epsilon^k  \right) \\
	&\leq 2\exp{-\frac{ \delta}{C^2_B\epsilon^{2-2k} }}.
	\end{align*}	
	Thus, we have
	\begin{align}
	\label{S4part1}
	\mathbb{P}&\left( \sup_{r\in [s,s+\delta] } \abs{M_{s,r}}^2\leq\epsilon^2, \int_{s}^{s+\delta}M_{s,l}^2 \abs{v^T A_i(X_l)}^2 dl\leq \delta C^2_B \epsilon^2 , \sup_{r\in [s,s+\delta] } \abs{\int_{s}^{r}M_{s,l}\cdot  v^T A_i(X_l) dW^i_l} \leq\delta \epsilon^k  \right)\nonumber \\
	&\geq \mathbb{P}\left( \sup_{r\in [s,s+\delta] } \abs{M_{s,r}}^2 \leq\epsilon^2\right)- 2\exp{-\frac{ \delta }{C^2_B\epsilon^{2-2k} }}.
	\end{align}
	On the other hand, by using \eqref{First implication} and \eqref{Ito} we see that
	\begin{align}
	\label{S4part2}
	&\left\{  \sup_{r\in [s,s+\delta] } \abs{M_{s,r}}^2\leq\epsilon^2, \int_{s}^{s+\delta}M_{s,l}^2 \abs{v^T A_i(X_l)}^2 dl\leq \delta C^2_B \epsilon^2 , \sup_{r\in [s,s+\delta] } \abs{\int_{s}^{r}M_{s,l}\cdot  v^T A_i(X_l) dW^i_l} \leq\delta \epsilon^k  \right\}\nonumber \\
	&\Rightarrow\left\{ \sup_{r\in [s,s+\delta] }\abs{M_{s,r}}^2\leq\epsilon^2, \sup_{r\in [s,s+\delta] }  \sum_{i=1}^{d} \int_{s}^{r}  \abs{v^T A_i(X_l)}^2dl\leq\epsilon^2+ 2C_B\delta \epsilon+d\delta \epsilon^k \right\}.
	\end{align}
	From \eqref{S4part1} and \eqref{S4part2} we can deduce that
	\begin{align}
	\mathbb{P}&\left( \sup_{r\in [s,s+\delta] }\abs{M_{s,r}}^2\leq\epsilon^2, \sup_{r\in [s,s+\delta] }  \sum_{i=1}^{d} \int_{s}^{r}  \abs{v^T A_i(X_l)}^2dl\leq\epsilon^2+ 2C_B\delta \epsilon+d\delta \epsilon^k  \right)\nonumber \\
	&\geq \mathbb{P}\left(\sup_{r\in [s,s+\delta] }\abs{M_{s,r}}^2\leq\epsilon^2\right)- 2\exp{-\frac{ \delta }{C^2_B\epsilon^{2-2k} }}.\label{Exponential}
	\end{align}
	By \eqref{Ellipticity}, there exists $\epsilon_1>0$ such that when $\epsilon<\epsilon_1$ we have
	\begin{align}
		\label{optimize}
	\sup_{r\in [s,s+\delta] }  \sum_{i=1}^{d} \int_{s}^{r}  \abs{v^T A_i(X_l)}^2dl= \langle M^m_{s,r} , M^m_{s,r} \rangle_{s,s+\delta}\geq \delta C_1 >\epsilon^2+ 2C_B\delta \epsilon+d\delta \epsilon^k.
	\end{align}
	Combining \eqref{Exponential} and \eqref{optimize} gives
	\begin{align*}
	&\mathbb{P}\left(\sup_{r\in [s,s+\delta] }\abs{M_{s,r}}\leq\epsilon,\ \langle M^m_{s,r} , M^m_{s,r} \rangle_{s,s+\delta}\geq \delta C_1\right)\\
	\leq &\mathbb{P}\left( \sup_{r\in [s,s+\delta] }\abs{M_{s,r}}^2\leq\epsilon^2,\sup_{r\in [s,s+\delta] }  \sum_{i=1}^{d} \int_{s}^{r}  \abs{v^T A_i(X_l)}^2dl>\epsilon^2+ 2C_B\delta \epsilon+d\delta \epsilon^k  \right)\\
	\leq &  2\exp{-\frac{ \delta }{C^2_B\epsilon^{2-2k} }}.
	\end{align*}
	Up to this point, all of our computations are done with \(v \) fixed. We shall conclude with a compactness argument. The idea is similar to that of \cite{MR942019} (page 127). Observe that if \(\sup_{t-s \leq \delta} \abs{X_{s,t}} \) is uniformly bounded, then \( \sup_{t-s \leq \delta}  \abs{v^T X_{s,t}} \) is Lipschitz as a function of \(v \). Moreover, since the unit ball of \(\mathbb{R}^d \) is compact, we can cover it with balls of radius $a<1$ and the number of these ball can be chosen to be less than $C/a^d$ for some constant $C$. 
	So, we have for any $\tau\in(0,\frac{1}{2})$
	\begin{align*}
		\mathbb{P}\left(\inf_{\norm{v}= 1} \sup_{t-s \leq \delta}  \abs{v^T X_{s,t}} \leq \epsilon \right)& \leq \frac{C_2 \theta^d}{\epsilon^d}\sup_{\norm{v}=1} \mathbb{P}\bigg( \sup_{t-s \leq \delta}  \abs{v^T X_{s,t}} \leq 2\epsilon \bigg)+\mathbb{P}\bigg(\sup_{t-s \leq \delta} \abs{X_{s,t}}>\theta \bigg)\\
		& \leq \frac{C_2 \theta^d}{\epsilon^d}\exp{-\frac{ \delta }{C^2_B\epsilon^{2-2k} }}+C_4\exp{-C_5\theta^{\tau}}.
	\end{align*}
	We used the exponential integrability of $\sup_{s\in[0,1]}\abs{X_s}$ in the last inequality (see for example proposition 2.9 of \cite{MR3531675}). Finally, set \[ \theta=\frac{\sigma^{\frac{1}{\tau}}}{\epsilon^{\frac{2-2k}{\tau}}},\]
	then we can find $\epsilon_0\leq \epsilon_1$ such that for $\epsilon<\epsilon_0$
	\begin{align*}
		\mathbb{P}\left(\inf_{\norm{v}= 1} \sup_{t-s \leq \delta}  \abs{v^T X_{s,t}} \leq \epsilon \right)& \leq \frac{C_2 \delta^{\frac{d}{\tau}}}{\epsilon^{\frac{(2-2k+\tau)d}{\tau}}}\exp{-\frac{ \delta}{C^2_B\epsilon^{2-2k} }}+C_4\exp{-C_5\frac{ \delta }{\epsilon^{2-2k} }}\\
		&\leq C_6\exp{-C_7\frac{ \delta }{\epsilon^{2-2k} }}.
	\end{align*}

\end{proof}

\begin{proposition}
	\label{Diffusion Holder}
	Under the assumptions of the previous lemma, for any \(\theta>\frac{1}{2} \), and \(k\in(0,1) \) such that \(\theta>\frac{1}{2-2k} \), we have 
	\begin{equation*}
	\mathbb{P}(L_{\theta}(X)<\epsilon)\leq C_1\exp(-C_2\epsilon^{2k-2}).
	\end{equation*}
In particular, $L^{-1}_\theta(X)\in L^P(\Omega)$, for any \(p\geq 1 \).
\end{proposition}
\begin{proof}
	Let us define
	\begin{equation*}
		D_{\theta}(X):=\inf_{\norm{v}= 1} \inf_{n\geq 1} \inf_{l\leq 2^n}\sup_{s,t\in I_{l,n}}\frac{ \abs{v^T X_{s,t}} }{2^{-n\theta}};
	\end{equation*}
	and
	\begin{equation*}
		\hat{D}_{\theta}(X):=\inf_{\norm{v}= 1} \inf_{n\geq 1} \inf_{l\leq 2^n}\sup_{s=\frac{l}{2^n},t\in I_{l,n}}\frac{ \abs{v^T X_{s,t}} }{2^{-n\theta}},
	\end{equation*}
	where
	\begin{equation*}
		I_{l,n}=[\frac{l}{2^n},\frac{l+1}{2^n} ].
	\end{equation*}
	Obviously we have $D_{\theta}(X)\geq \hat{D}_{\theta}(X)$. The exact same argument of lemma 3 of \cite{MR3112925} can be applied here, from which we can deduce 
	
	\[ L_{\theta}(X)>\frac{1}{2\cdot 8^\theta}D_{\theta}(X)\geq \frac{1}{2\cdot 8^\theta}\hat{D}_{\theta}(X) . \]
	Thus, it suffices to give estimate for $\hat{D}_{\theta}(X)$. By definition, we have
	\begin{equation*}
	\mathbb{P}(\hat{D}_{\theta}(X)<\epsilon)\leq\sum_{n=1}^{\infty}\sum_{k=1}^{2^n-1 }\mathbb{P}(\inf_{\norm{v}= 1} \sup_{s=\frac{l}{2^n},t\in I_{l,n}}\frac{\abs{v^T X_{s,t}} }{2^{-n\theta}}<\epsilon ).
	\end{equation*}
	When $\epsilon$ is sufficiently small, we can apply the previous lemma to get
	\begin{equation*}
	\mathbb{P}(\hat{D}_{\theta}(X)<\epsilon)\leq C_1\sum_{n=1}^{\infty}2^n\exp{-C_2\epsilon^{2k-2}2^{n(\theta(2-2k)-1) }}.
	\end{equation*}
	Since \(\theta>\frac{1}{2-2k} \), we can find \(C_3,C_4>0 \) uniformly over $\epsilon\leq 1,\; n\geq 1$, such that
	\begin{equation*}
	2^n\exp{-C_2\epsilon^{2k-2}2^{n(\theta(2-2k)-1) }}\leq \exp{C_3-C_4n\epsilon^{2k-2} }.
	\end{equation*}
	Thus,
	\begin{equation*}
	\mathbb{P}(\hat{D}_{\theta}(X)<\epsilon)\leq C_1\sum_{n=1}^{\infty}\exp{C_3-C_4n\epsilon^{2k-2} }\leq C_5\exp{-C_6\epsilon^{2k-2} },
	\end{equation*}
	and the proof is finished.
	
\end{proof}

\subsection{Non-degenerate property of Jacobian processes}
Now, we move on to the study of non-degenerate property of the Jacobian process of \(X_t\). For the sake of conciseness, we adopt the notions from \cite{MR2786645} and introduce the following
\begin{definition}
	A family of sets \(\{E_\epsilon \}_{\epsilon\in[0,1]}\in \mathcal{F} \) is said to be ``almost true" if for any \(p\geq 1\), we can find \(C_p>0\) such that
	\begin{equation*}
	\mathbb{P}(E_\epsilon)\geq1- C_{p}\epsilon^p.
	\end{equation*}
	Similarly for ``almost false". Given two
	such families of events A and B, we say that ``A almost implies B" and we write $A\Rightarrow_\epsilon B$ if \(A\setminus B \) is almost false.
\end{definition}
 It is straightforward to check that these ``almost'' implications are transitive and invariant under any reparametrisation of the form $\epsilon\mapsto \epsilon^\alpha$ for $\alpha>0$.
\begin{remark}
	A typical situation where this definition naturally appears is as follows. Suppose that $Z$ is a random variable in some probability space such that $\mathbb{E}\abs{Z}^p<\infty$ for any $p\geq 1$, then  
	\begin{equation*}
		\mathbb{P}\left(\abs{Z}>\frac{1}{\epsilon}\right)\leq \mathbb{E}\abs{Z}^p\epsilon^{p}.
	\end{equation*}
	In other words, $\left\{ \abs{Z}\leq \frac{1}\epsilon \right\}$ is almost true. We will simply write it as $\abs{Z}\leq_\epsilon \frac{1}\epsilon$.
\end{remark}
Our next result establishes a non-degenerate property of Jacobian process \(J^X\).

\begin{proposition}
	\label{Jacobian Roughness}
Let \(f(s)\in\mathbb{R}^d\) be a diffusion controlled by $J^X$, such that for any $\gamma\in (0,\frac{1}{2})$ we have $\norm{f}_\gamma\in L^p(\Omega)$ for any $p\geq 1$. Then
\begin{equation*}
	\left\{ \sup_{r\in[0,t]}  \abs{\int_{0}^{r}f(s)^Td\boldsymbol{J^X_{s\leftarrow 0}}}    \leq \epsilon  \right\}\Rightarrow_\epsilon \left\{ \sup_{r\in[0,t]}\abs{f(r)}\leq C\epsilon^{\alpha}  \right\},
\end{equation*}
for some constant $C>0$ and $\alpha\in (0,1)$.

\end{proposition}

\begin{proof}
	Since $J^X$ is a semi-martingale, we can consider
	\begin{equation*}
		M_t:=\int_{0}^{t}f(s)^T\circ dJ^X_{s\leftarrow 0}. 
	\end{equation*}
	One has the following representation (recall $f^m$ means the local martingale part of $f$).
	\begin{equation*}
		M_r=\sum_{i=1}^{d} \int_{0}^{r} f(s)^T \cdot DA_i(X_s)J^X_{s\leftarrow 0}dW^i_s+\int_{0}^{r}f(s)^T \cdot DB(X_s)J^X_{s\leftarrow 0} ds+\frac{1}{2}\langle f^m , M^m \rangle_{[0,r]}.
	\end{equation*}
	By our assumptions and proposition \ref{Integrability}, it is easy to see $M_r$ verifies the assumption of lemma 4.11 of \cite{MR2838095}, from which we have
	\begin{equation}
		\label{Classical Norris' lemma}
		\left\{\sup_{r\in[0,t]}\abs{M_r}\leq \epsilon \right\}\Rightarrow_\epsilon \left\{\max_{1\leq i\leq d}\sup_{r\in[0,t]} \abs{f(s)^T \cdot DA_i(X_s)J^X_{s\leftarrow 0}}<\epsilon^{\alpha'} \right\}.
	\end{equation}
for some $\alpha'\in (0,1)$. Moreover, since
\begin{equation*}
	\abs{J^X_{0 \leftarrow s}}=\left\{ \inf_{\abs{v}=1}\abs{J^X_{s \leftarrow 0}\cdot v}   \right\}^{-1},
\end{equation*}
we have, by proposition \ref{Integrability}, for any \(\eta>0 \) that
\begin{equation*}
	\mathbb{P}\left( \inf_{s\in[0,1]}\abs{v^T \cdot J^X_{s\leftarrow 0}}<\epsilon^\eta \abs{v}\bigg)=\mathbb{P}\bigg(\norm{\Phi}_\infty>\frac{1}{\epsilon^\eta}  \right)\leq C_p\epsilon^{\eta p}.
\end{equation*}
Thus, for \(s\in [0, t] \) we have 
\begin{equation}
	\label{Almost sure of eta}	
	\sum_{i=1}^{d}\abs{ f(s)^T\cdot DA_i(X_s)J^X_{s\leftarrow 0}}^2\geq_\epsilon \epsilon^{2\eta}\sum_{i=1}^{d}\abs{ f(s)\cdot DA_i(X_s)}^2 .
\end{equation}
Finally, \textbf{\emph{Assumption}} \ref{Assumption 3} gives
\begin{equation}
	\label{Where Assumption 3 is used}
	\sum_{i=1}^{d}\abs{f(s)^T \cdot DA_i(X_s)}^2\geq C_J \abs{f(s)}^2.
\end{equation}
Combining \eqref{Classical Norris' lemma} \eqref{Almost sure of eta} \eqref{Where Assumption 3 is used} and use the fact that $\eta$ is arbitrary, we deduce
\begin{equation*}
	\left\{ \sup_{r\in[0,t]}  \abs{\int_{0}^{r}f(s)^Td\boldsymbol{J^X_{s\leftarrow 0}}}   \leq \epsilon  \right\}\Rightarrow_\epsilon \left\{ \sup_{r\in[0,t]}\abs{f(r)}\leq C\epsilon^{\alpha}  \right\}
\end{equation*}
for some $\alpha \in (0,1)$.

\end{proof}

With the previous proposition, we can show the desired non-degenerate property of $D_rX_s$ as a family of processes index by $r$. 
\begin{proposition}
	\label{last step}
	Let \(f(s)\in\mathbb{R}^{d} \) be a diffusion controlled by $D_rX$ for any $r\in [0,1]$ and that for any $\gamma\in (0,\frac{1}{2})$, we have $\norm{f}_\gamma\in L^p(\Omega)$ for all $p\geq 1$. Then for any $t\in[0,1]$ we have
	\begin{equation*}
		\left\{\sup_{r\in[0,t]}\abs{  \int_{0}^{t}f(s)^Td\boldsymbol{D}_r\boldsymbol{X}_s}\leq\epsilon \right\}\Rightarrow_\epsilon \sup_{s\in[0,t]}\abs{f(s)}\leq \epsilon^\alpha.
	\end{equation*}
\end{proposition}
\begin{proof}
	As before, we can write
	\begin{equation*}
		\int_{0}^{t}f(s)^T\boldsymbol{D}_r\boldsymbol{X}_s=\int_{r}^{t}f(s)^T\circ d J^X_{s\leftarrow 0} \cdot J^X_{0 \leftarrow r}A(X_r).
	\end{equation*}
	Since $A$ is an elliptic system, we see immediately
	\begin{equation}
		\label{Part 1 of last step}
		\abs{\int_{r}^{t}f(s)^T\circ d J^X_{s\leftarrow 0} \cdot J^X_{0 \leftarrow r}A(X_r)}\geq C_1 \abs{\int_{r}^{t}f(s)^T\circ d J^X_{s\leftarrow 0} \cdot J^X_{0 \leftarrow r}}.
	\end{equation}
	Moreover, we have
	\begin{equation*}
		\abs{J^X_{0 \leftarrow r}}=\left\{ \inf_{\abs{v}=1}\abs{J^X_{r \leftarrow 0}\cdot v}   \right\}^{-1},
	\end{equation*}
	which implies that for any $\beta>0$ 
	\begin{equation*}
		\mathbb{P}\left( \inf_{\abs{v}=1}\abs{J^X_{r \leftarrow 0}\cdot v}\leq \epsilon^\beta  \right)\leq \mathbb{P}(\norm{\Phi}_\infty \geq \frac{1}{\epsilon^\beta} )\leq C_p \epsilon^p.
	\end{equation*}
	By taking transpose, we have that
	\begin{equation}
		\label{Part 2 of last step}
		\abs{\int_{r}^{t}f(s)^T\circ d J^X_{s\leftarrow 0} \cdot J^X_{0 \leftarrow r}}\geq_\epsilon \epsilon^\beta \abs{\int_{r}^{t}f(s)^T\circ d J^X_{s\leftarrow 0} }.
	\end{equation}
	Combining \eqref{Part 1 of last step} and \eqref{Part 2 of last step} gives
	\begin{equation*}
		\abs{\int_{0}^{t}f(s)^T\boldsymbol{D}_r\boldsymbol{X}_s}\geq_\epsilon C_1 \epsilon^\beta \abs{\int_{r}^{t}f(s)^T\circ d J^X_{s\leftarrow 0}}.
	\end{equation*}
	As a result, we have for any $p\geq 1$ that
	\begin{equation*}
		\left\{ \sup_{r\in[0,t]} \abs{\int_{0}^{t}f(s)^T\boldsymbol{D}_r\boldsymbol{X}_s}\leq \epsilon \right\}\Rightarrow_\epsilon \left\{ \sup_{r\in[0,t]} \abs{\int_{r}^{t}f(s)^T\circ d J^X_{s\leftarrow 0}} \leq C_1 \epsilon^{1-\beta}  \right\}\Rightarrow \left\{\sup_{r\in[0,t]} \abs{\int_{0}^{r}f(s)^T\circ d J^X_{s\leftarrow 0}} \leq C_2 \epsilon^{1-\beta}  \right\}.
	\end{equation*}
	From proposition \ref{Jacobian Roughness}, we deduce
	\begin{equation*}
		\left\{ \sup_{r\in[0,t]} \abs{\int_{0}^{r}f(s)^T\circ d J^X_{s\leftarrow 0}} \leq C_2 \epsilon^{1-\beta}   \right\}\Rightarrow_\epsilon \left\{  \sup_{r\in[0,t]}\abs{f(r)}\leq C_3\epsilon^{(1-\beta)\alpha'} \right\}
	\end{equation*}
	for some $\alpha'\in (0,1)$. Since $\beta$ is arbitrary and almost true implications are transitive, our result follows.

\end{proof}

We prepare another lemma for next section.
\begin{lemma}
	\label{Holder regularity}
	If \(f(s)\in\mathbb{R}^{d\times d} \) is a diffusion process controlled by \(D_rX_s\) such that 
	\begin{equation*}
		\norm{f}_\infty,\ \norm{f^M}_\infty\in L^p(\Omega),\ \forall p\geq 1.
	\end{equation*}
	Then
	\begin{equation*}
		F(r)=\int_{0}^{t}f(s)d\boldsymbol{D}_r\boldsymbol{X}_s
	\end{equation*}
	is \(\gamma \)-H\"older continuous for any \(\gamma<\frac{1}{2} \) and \(\norm{F}_\gamma\in L^p(\Omega) \) for all \(p\geq 1 \). 
\end{lemma}
\begin{proof}
	We have
	\begin{equation*}
		F(r)=\int_{0}^{t}f(s)d\boldsymbol{D}_r\boldsymbol{X}_s=\int_{0}^{t}f(s)\circ dD_rX_s=\int_{r}^{t}f(s)\circ dD_rX_s\ a.s.
	\end{equation*}
	where we used the property that $D_rX_s=0$ a.s. for $s<r$.
	As a result, for $0\leq u\leq v\leq t$ we have
	\begin{equation*}
		F(u)-F(v)=\int_{v}^{t}f(s)\circ d(D_uX_s-D_vX_s)+\int_{u}^{v}f(s)\circ dD_uX_s.
	\end{equation*} 
	The first term on the right hand side can be written as
	\begin{equation*}
		\Gamma_1=\int_{v}^{t}f(s)\circ dJ^X_{s\leftarrow 0} \cdot (J^X_{0\leftarrow u}A(X_u)-J^X_{0\leftarrow v}A(X_v) ).
	\end{equation*}
	Similarly, We can write the second term as
	\begin{equation*}
		\Gamma_2=\int_{u}^{v}f(s)\circ dJ^X_{s\leftarrow 0} \cdot J^X_{0\leftarrow u}A(X_u).
	\end{equation*}
	For $\Gamma_1$ we have for any $\alpha\in (0,\frac{1}{2})$ and $p\geq 1$
	\begin{align*}
		\mathbb{E}\abs{\Gamma_1}^p&\leq 2^p \cdot 
		\mathbb{E}\left\{   \sup_{t\in[0,1]}\abs{\int_{0}^{t}f(s)\circ dJ^X_{s\leftarrow 0}}^p\left(C_B \norm{\Phi}_\alpha (u-v)^\alpha+\norm{\Phi}_\infty C_B (v-u)^\alpha \right)^p\right\}\\
		&\leq 2^{2p-1} \cdot  \mathbb{E}\left\{ \sup_{t\in[0,1]}\abs{\int_{0}^{t}f(s)\circ dJ^X_{s\leftarrow 0}}^p (C_B+\norm{\Phi}_\infty+\norm{\Phi}_\alpha)^p        \right\} \abs{v-u}^{\alpha p} .
	\end{align*}
	For $\Gamma_2$ we have for any $p\geq 1$
	\begin{align*}
		\mathbb{E}\abs{\Gamma_2}^p\leq \mathbb{E}\left\{C_B^p \norm{\Phi}^p_\infty \abs{\int_{u}^{v}f(s)\circ dJ^X_{s\leftarrow 0}}^p         \right\}
		\leq C_B^p\mathbb{E}(\norm{\Phi}^{2p}_\infty)^{\frac{1}{2}}\mathbb{E}\left\{\abs{\int_{u}^{v}f(s)\circ dJ^X_{s\leftarrow 0}}^{2p}       \right\}^{\frac{1}{2}}.
	\end{align*}
	By Burkholder-Davis-Gundy inequality, standard computation gives that for any $q\geq 2$
	\begin{align*}
		\mathbb{E}\left\{\abs{\int_{u}^{v}f(s)\circ dJ^X_{s\leftarrow 0}}^{q}       \right\}&\leq\mathbb{E}\left\{ \abs{\int_{u}^{v}f(s) d(J^X_{s\leftarrow 0})^m+\int_{u}^{v}f(s)d(J^X_{s\leftarrow 0})^b+\frac{1}{2}\langle f^m, (J^X_{s\leftarrow 0})^m \rangle_{u,v} }^q          \right\}\\
		&\leq C_q \mathbb{E}\left\{  \norm{f}_\infty^q\norm{\Phi}_{\infty}^qC_B^q\abs{v-u}^{\frac{q}{2}}+\norm{f}_\infty^q\norm{\Phi}_{\infty}^qC_B^q\abs{v-u}^{q}\right\}\\
		&+C'_q(\mathbb{E}\norm{f^m}_\infty^{2q})^{\frac{1}{2}}(\mathbb{E}\norm{\Phi}_{\infty}^{2q}C_B^{2q})^{\frac{1}{2}}\abs{v-u}^{\frac{q}{2}}\\
		&\leq P_q( \norm{f}_\infty, \norm{f^m}_\infty, \norm{\Phi}_{\infty}, C_B,) \abs{v-u}^{\alpha q},
	\end{align*}
	where $P_q$ is some polynomial depends on $q$. It is also relative easy to check that, under our assumptions
	\[ \sup_{t\in[0,1]}\abs{\int_{0}^{t}f(s)\circ dJ^X_{s\leftarrow 0}}\in L^p(\Omega),\ \ \forall p\geq 1.          \]
	Hence, combining our estimate for $\Gamma_1$ and $\Gamma_2$ gives
	\begin{equation*}
		\mathbb{E}\abs{F(v)-F(u)}^p\leq 2^{p-1}(\mathbb{E}\abs{\Gamma_1}^p+\mathbb{E}\abs{\Gamma_2}^p )\leq C_p \abs{v-u}^{\alpha p}.
	\end{equation*}
	We conclude with Kolmogorov continuity theorem and Besov–H\"older embedding (see theorem A.10 of \cite{friz2010multidimensional}).
\end{proof}



\section{Existence of smooth density and Gaussian type upper bound}

\subsection{Malliavin smoothness and integrability}

Let \( X_t,Y_t \) be the processes defined in theorem \ref{main1}. We need to show $Y_t\in \mathbb{D}^\infty$ before we can apply proposition \ref{Smoothness}. In general, it is not an easy task to show the Malliavin smoothness of solution to a rough differential equation (see \cite{MR3298472}, proposition 7.5 and the note after). This result has been obtained for Gaussian rough paths in \cite{MR3229800}, but the technique used there is very difficult to generalize to non-Gaussian rough paths. It would be interesting to investigate further in this direction. 

For our purpose, however, we can again use the fact that rough integrals against $\boldsymbol{X_t}$ coincide with Stratonovich integrals against $X_t$ to our advantage. Indeed, the couple process \( (X_t,Y_t ) \) is solution to a stochastic differential equation driven by Brownian motion. Since the vector fields \( \{A_i\}_{1\leq i\leq d},B,\{V_i \}_{0\leq i\leq d}\in C^\infty_b(\mathbb{R}^d)  \), we immediately have $(X_t,Y_t)\in \mathbb{D}^\infty$. The Jacobian process of \((X_t,Y_t) \), is given by 
\begin{equation*}
J^{X,Y}_{t\leftarrow 0}=\begin{pmatrix}
\frac{\partial X_t}{\partial X_0} & \frac{\partial X_t}{\partial Y_0}\\[6pt]
\frac{\partial Y_t}{\partial X_0} & \frac{\partial Y_t}{\partial Y_0}
\end{pmatrix}=\begin{pmatrix}
J^X_{t\leftarrow 0} & 0\\[6pt]
\frac{\partial Y_t}{\partial X_0}  &J^Y_{t\leftarrow  0}
\end{pmatrix}
\end{equation*} 
with inverse 
\begin{equation*}
J^{X,Y}_{0\leftarrow t}=\begin{pmatrix}
J^X_{0\leftarrow t} & 0\\[6pt]
-J^Y_{0\leftarrow  t} \frac{\partial Y_t}{\partial X_0}J^X_{0\leftarrow t}  &J^Y_{0\leftarrow  t}
\end{pmatrix}.
\end{equation*}
Define
\begin{equation*}
	U_t=(X_t,Y_t, J^{X,Y}_{t\leftarrow 0},J^{X,Y}_{0\leftarrow t} ), 
\end{equation*}
 then by proposition \ref{Integrability}, we know for any \(0< \gamma<\frac{1}{2} \) the \(\gamma \)-H\"older constant \(\norm{U_t}_\gamma\in L^p(\Omega) \) for all \(p\geq1 \). Since \(Y_t\) is solution to a rough differential equation driven by \(X_t\) with \(C^\infty_b\) vector fields, \(Y_t\) is automatically a rough path controlled by \(X_t\). Along the same lines of proposition 8.1 and corollary 8.2 of \cite{MR3298472}, we know \(\norm{Y ,Y'}_{X, 2\gamma} \) is in \(L^p(\Omega) \) for any \(p\geq1 \). So we have
\begin{proposition}
	\label{integrability}
	Under the assumptions of theorem \ref{main1}. For any $\gamma\in(0,\frac{1}{2})$, define
	\begin{equation*}
	\mathcal{L}_{\theta }=:1+L_\theta(X)^{-1}+\norm{M_t}_\gamma+\norm{Y ,Y'}_{X, 2\gamma}+\rho_{\gamma}(\boldsymbol{X}),
	\end{equation*}
	 then \(\mathcal{L}_{\theta }\in L^p(\Omega) \) for all \(p\geq 1 \).
\end{proposition}

\subsection{Proof of main results}

In order to prove theorem \ref{main1}, by lemma 2.31 of \cite{MR2200233} and proposition \ref{Smoothness}, it boils down to get an estimate on \[ \mathbb{P}\left(\inf_{\norm{v}= 1} v^T \Gamma(Y_t)v   <\epsilon\right) \] for every \(t\in (0,1]\). Note that, we can write
\[\Gamma(Y_t)=J^Y_{t\leftarrow 0}\cdot C(Y_t) \cdot (J^Y_{t\leftarrow 0})^T,   \]
  where \(C(Y_t) \) is the so called reduced Malliavin matrix of \(Y_t \). Since $J^Y_{t\leftarrow 0}$ already verifies proposition \ref{Smoothness}, it suffices to prove our estimate for $C(Y_t)$.  
  
  For a fixed unit vector \(v\in\mathbb{R}^d \), we define
\begin{equation*}
f_v^i(s)=v^T\cdot J^Y_{0\leftarrow s}V_i(Y_{s}).
\end{equation*}
With this and our previous computation \eqref{Malliavin derivative}, we have
\begin{align*}
 v^T C_t(Y)v  =\sum_{j=1}^{d} \int_{0}^{1}\abs{\sum_{i=1}^{d}  \int_{0}^{t}f_v^i(s)d\boldsymbol{D}^j_r\boldsymbol{X}^i_s}^2dr.
\end{align*}
\begin{proof}[Proof of theorem \ref{main1}.]
 It is easily checked that \(f_v^i(s)\) satisfy the assumptions of lemma \ref{Holder regularity}. So for any $\gamma\in (0,\frac{1}{2})$, let
 \begin{equation*}
 	F^{j,i}(r)=\sum_{i=1}^{d}  \int_{0}^{t}f_v^i(s)d\boldsymbol{D}^j_r\boldsymbol{X}^i_s,
 \end{equation*}
 we have $\mathcal{R}=\norm{F}_\gamma\in L^p(\Omega)$ for all $p\geq 1$. When $v^T C_t(Y)v<1$, by lemma A.3 of \cite{MR2814425} we have
\begin{align}
	\label{Holder L^2}
	\sup_{r\in[0,1]}\abs{\sum_{i=1}^{d}  \int_{0}^{t}v^T J^Y_{0\leftarrow s}V_i(Y_s)d\boldsymbol{D}^j_r\boldsymbol{X}^i_s}&\leq 2\mathcal{R}^{\frac{2\gamma}{2\gamma+1}} (v^T C_t(Y)v)^{\frac{1}{2\gamma+1}}+2(v^T C_t(Y)v)\nonumber \\
	&\leq (2\mathcal{R}^{\frac{2\gamma}{2\gamma+1}}+2)(v^T C_t(Y)v)^{\frac{1}{2\gamma+1}}.
\end{align}
We can deduce from \eqref{Holder L^2} that there exists some constant $C>0$ such that 
\begin{equation*}
	\sup_{r\in[0,t]} \abs{\int_{0}^{t}\Xi_s d\boldsymbol{D}_r\boldsymbol{X}_s}\leq C (2\mathcal{R}^{\frac{2\gamma}{2\gamma+1}}+2)(v^T C_t(Y)v)^{\frac{1}{2\gamma+1}},
\end{equation*}
where $\Xi_s\in\mathbb{R}^{d\times d}$ with $(\Xi_s)_{ji}=f_v^i(s)$. Now by proposition \ref{last step} we are able to find \(\alpha>0 \) such that
\begin{equation}
	\label{step 1}
	\left\{v^T C_t(Y)v<\epsilon  \right\}\Rightarrow_\epsilon \left\{ \sup_{s\in[0,t]}\abs{f_v^i(s)}\leq( 2\mathcal{R}^{\frac{2\gamma}{2\gamma+1}}+2)^\alpha \epsilon^{\frac{\alpha}{2\gamma+1}} \right\}. 
\end{equation} 
The key observation is that
\begin{equation*}
	f_v^i(t)=v^TJ^Y_{0\leftarrow t}V_i(Y_t)=v^TV_i(y_0)+\int_{0}^{t}[V_i, V_0](Y_s)ds+\sum_{j=1}^{d}\int_{0}^{t}[V_j,V_i](Y_s)d\boldsymbol{X}^i_s.
\end{equation*}
By proposition \ref{Diffusion Holder} and \ref{Norris}, there exist some \(q,l>0 \) 
\begin{equation*}
	\norm{[V_i, V_0]}_{\infty}\; \&\; \; \norm{[V_j,V_i]}_{\infty}\leq M\mathcal{L}_{\theta }^q\norm{f_v}_{\infty}^l.
\end{equation*}
By induction, we can see
\begin{equation*}
	\norm{v^T J^Y_{0\leftarrow s}W(Y_t)}_{\infty}\leq C \mathcal{L}_{\theta }^{m(k)}\norm{f_v}_{\infty}^{n(k)}
\end{equation*}
for all \(W\in \mathcal{W}_k \), where \(m(k), n(k) \) are constants only depend on \(k\).
Since \(\{V_i \}_{0\leq i\leq d}  \) satisfy the parabolic H\"{o}rmander's condition, we can find \(a_0>0 \) such that
\begin{align}
	\label{step 2}
	a_0= \inf_{\norm{v}= 1} \sum_{W\in \cup_{0\leq k\leq k_0} \mathcal{W}_k} \abs{v^T W(y_0)}& \leq\inf_{\norm{v}= 1}\sum_{W\in  \mathcal{W}_k} \norm{v^T J^Y_{0\leftarrow s}W(Y_s)}_{\infty}\nonumber \\
	&\leq\inf_{\norm{v}= 1}\norm{v^T J^Y_{0\leftarrow s}W(Y_t)}_{\infty}\nonumber \\
	&\leq \inf_{\norm{v}= 1} C \mathcal{L}_{\theta }^{m(k)}\norm{f_v}_{\infty}^{n(k)}.
\end{align}
From \eqref{step 1} and \eqref{step 2} we have
\begin{equation}
	\label{Step 3.1}
	\left\{\inf_{\norm{v}= 1}v^T C_t(Y)v<\epsilon  \right\}\Rightarrow_\epsilon \left\{ a_0\leq C \mathcal{L}_{\theta }^{m(k)}( 2\mathcal{R}^{\frac{2\gamma}{2\gamma+1}}+2)^{\alpha n(k)} \epsilon^{\frac{\alpha n(k)}{2\gamma+1}}  \right\}.
\end{equation}
On the other hand, since $\mathcal{L}_{\theta }, \mathcal{R}\in L^p(\Omega)$ for any $p\geq 1$, we have
\begin{equation}
	\label{Step 3.2}
	\mathbb{P}\left\{ a_0\leq C \mathcal{L}_{\theta }^{m(k)}( 2\mathcal{R}^{\frac{2\gamma}{2\gamma+1}}+2)^{\alpha n(k)} \epsilon^{\frac{\alpha n(k)}{2\gamma+1}}   \right\} \leq C_p\epsilon^p.
\end{equation}
Hence, combining \eqref{Step 3.1} \eqref{Step 3.2}, we see that when $\epsilon$ is sufficiently small
\begin{equation}
	\label{Main proof 2}
	\mathbb{P}\left(\inf_{\norm{v}=1} v^T C_t(Y)v\leq \epsilon  \right)\leq C_p \epsilon^p.
\end{equation}
This finishes the proof of existence of a smooth density.

Let $p_{Y_t}(y)$ be the density of $Y_t$, we have the following upper bound (see proposition 2.1.4 and 2.1.5 in \cite{MR2200233})
\begin{equation*}
	p_{Y_t}(y)\leq C\cdot \mathbb{P}\left(\sup_{s\in[0,t]} \abs{Y_s-y_0}>\abs{y-y_0}\right)^{\frac{1}{2}} \norm{\det(\Gamma(Y_t))^{-1} }^m_{L^k(\Omega)} \norm{DY_t}^l_{h,\rho},
\end{equation*}
for some constants \(m,l,h,\rho,k\). Consider \( (X_t,Y_t) \) as solution to a stochastic differential equation driven by Brownian motion, then by proposition 2.10 from \cite{MR3531675} we have
\begin{equation*}
	\mathbb{P}\left(\sup_{s\in[0,t]} \abs{Y_s-y_0}>\abs{y-y_0}\right)^{\frac{1}{2}}\leq\exp(-\frac{C (y-y_0)^2}{t}).
\end{equation*}
We have just proved that \( \det(\Gamma(Y_t))^{-1} \in L^p(\Omega) \) for all \(p\geq 1\). Finally, for \(\norm{DY_t}^l_{h,\rho} \) the exact argument of lemma 4.1 in \cite{MR3531675} applies to \( (X_t,Y_t) \). Therefore \(\norm{DY_t}^l_{h,\rho}<C(t) \), for some positive constant \(C(t) \), and the proof is finished.  	
\end{proof}


\appendix
\setcounter{secnumdepth}{0}

{\bf Acknowledgement:} The author is grateful to Fabrice Baudoin for many insightful discussions.

{\bf Declarations:} Partial financial support was received from National Science Foundation grant DMS-1901315.

\bibliographystyle{abbrv}
\bibliography{ref}

Guang Yang: \texttt{yang2220@purdue.edu}
\\
Department of Mathematics,
Purdue University,
West Lafayette, IN 47907

\end{document}